\documentclass[12pt]{amsart}
\usepackage{latexsym}
\usepackage{amsmath, amscd, amsfonts}
\usepackage[square, comma, sort&compress, numbers]{natbib}

{\theoremstyle{plain}
    \newtheorem{Thm}{\bf Theorem}[section]
    \newtheorem{Prop}[Thm]{\bf Proposition}
    
    \newtheorem{Lma}[Thm]{\bf Lemma}
    \newtheorem{Cor}[Thm]{\bf Corollary}
    \newtheorem{Q}[Thm]{\bf Question}

}
{\theoremstyle{remark}
    \newtheorem{Rem}[Thm]{\bf Remark}
    \newtheorem{Exa}[Thm]{\bf Example}
}
{\theoremstyle{definition}
    \newtheorem{Def}[Thm]{\bf Definition}
}
\numberwithin{equation}{section}

\DeclareFontFamily{OMS}{rsfs}{\skewchar\font'60}
\DeclareFontShape{OMS}{rsfs}{m}{n}{<-5>rsfs5 <5-7>rsfs7 <7->rsfs10 }{}
\DeclareSymbolFont{rsfs}{OMS}{rsfs}{m}{n}
\DeclareSymbolFontAlphabet{\scr}{rsfs}

\newcommand{\sC}{\scr{C}}
\newcommand{\sD}{\scr{D}}

\newcommand{\mlabel}[1]%
  {\mbox{}\marginpar{\raggedleft\hspace{0pt}{\rm\ttfamily#1}}\label{#1}}

\newcommand{\length}{\operatorname{\lambda}}

\newcommand{\Hom}[3]{\operatorname{Hom}_{#1}(#2,#3){}}

\newcommand{\fm}{{\mathfrak m}}

\newcommand{\fa}{{\mathfrak a}}

\newcommand{\fC}{{\mathcal C}}

\newcommand{\ringR}{\text{$(R,\fm,k)$ }}

\newcommand{\brq}{^{[q]}}

\newcounter{hours}\newcounter{minutes}

\newcommand{\excise}[1]{}
\begin{document}

\title{\bf Strong test ideals associated to Cartier algebras}

\author[F.~Enescu, I.~Ilioaea]{Florian Enescu, Irina Ilioaea}

\address{Department of Mathematics and Statistics, Georgia State University, Atlanta GA 30303}
\email{fenescu@gsu.edu, iilioaea1@gsu.edu}
\thanks{The first author was partially supported by the NSA grant H98230-12-1-0206.}
\date{}


\begin{abstract}
In this note, we use the theory of test ideals and Cartier algebras to examine the interplay between the tight and integral closures in a local ring of positive characteristic. Using work of Schwede, we prove the abundance of strong test ideals, recovering some older fundamental results, and use this approach in concrete computations. In the second part of the paper, the case of Stanley-Reisner rings is fully examined.
\end{abstract}
\maketitle

\section{Introduction}

In this note $\ringR$ denotes a local F-finite reduced ring of prime positive characteristic $p$. For a subset $I \subseteq R$ the notation $I \leq R$ will mean that $I$ is an ideal of $R$. Let $R^o$ denote the complement of the union of the minimal primes of $R$. Let $F^e : R \to R$ denote the $e$th iteration of the Frobenius map, that is, $F^e(r)=r^q$, where $q=p^e, e\in \mathbb{N}$. This defines a new $R$-module structure on $R$ denoted here by $R^{(e)}$. For an $R$-module $M$, let $F^e_R(M) = R^{(e)} \otimes_R M$. For a submodule $N \subseteq M$ we denote $N\brq_M = Im(F^e_R(N) \to F^e_R(M))$ and for $x \in M$, we let $x^q$ denote the image of $1 \otimes x$ in $F^e_R(M)$.

\begin{Def}
Let $I \leq R$. Let $N$ be a submodule of an $R$-module $M$. Then the {\it tight closure} of $I$ is the ideal 

$$I^* = \{ x \in R : {\rm there \ exists} \ c\in R^o {\rm \ such \ that} \ cx^q \in I\brq, {\rm for \ all} \ q=p^e \gg0\},$$and the tight closure of $N$ is $M$ is

$$N^* = \{ x \in M : {\rm there \ exists} \ c\in R^o {\rm \ such \ that} \ cx^q \in N\brq _M, {\rm for \ all} \ q=p^e \gg0\}.$$

\end{Def}

Tight closure theory is an important contribution to commutative algebra that has reshaped the understanding of many classical results in this area and produced new and important developments. Despite the large amount of work produced in connection to it, it is generally accepted that computing the tight closure of a given ideal in a particular ring can be very difficult. 

For an ideal $I$ in $R$, it is well known that $I^* \subseteq \overline{I}$, where $\overline{I}$ denotes the integral closure of $I$.  The purpose of this note is to further the study of the relationship between these ideal closures., by considering the following question: What elements of the integral closure of an ideal belong to its tight closure? This line of investigation was initiated by Huneke, in~\cite{Hu}, who introduced the concept of strong test ideal in relation to it. This notion provides interesting concrete information about the object of our study, as we explain below.

\begin{Def}
Let $T$ be an ideal of $R$ such that $T \cap R^o \neq \emptyset$. Then $T$ is a {\it strong test ideal} for $R$ if and only if $TI^*= TI$, for all ideals $I \leq R$. 
\end{Def}

Huneke has observed that the minimal number of generators of a strong test ideal $T$ is an upper bound for the minimal degree of an integral dependence equation that an element $x \in I^*$ satisfies over $I$, see Theorem 2.1 in~\cite{Hu}.

Strong test ideals are closely connected to the concept of test ideal, which plays a prominent role in the tight closure theory. We recall the definition here.

\begin{Def}

The {\it big test ideal} is defined as $\tau_b(R) = \cap_{M} {\rm Ann}_R(0^*_M)$. The {\it finitistic test ideal} of $R$ is $\cap _{I \leq R} I : I^*$ and is denoted by $\tau_{fg}(R)$. 
\end{Def}

A longstanding conjecture of tight closure theory states that $\tau_b(R) =\tau_{fg}(R)$. Several people have studied this type of ideals, see~\cite{E, HaS, Hu, V}. In a significant development connecting these ideals to the strong test ideal, Vraciu has shown that,  for a ring satisfying our hypotheses, natural strong test ideals do exist. More precisely, Vraciu's Theorems 3.1 and 3.2 in~\cite{V} (together with Theorem 5.1 in~\cite{AE}) show that $\tau_b(R)$ and, if $R$ is complete, $\tau_{fg}(R)$ are strong test ideals. Therefore, there exists an uniform upper bound for the minimal degree of an integral dependence equation that an element $x \in I^*$ satisfies over $I$. Unfortunately, practical applications of Vracius's work are not immediate, since the big test ideal and the finitistic test ideal are hard to compute in concrete examples. 

The purpose of the note is to highlight a computational approach to the problem of identifying an uniform upper bound for the minimal degree of an integral dependence equation for tight closure elements via strong test ideals. In the process, we update the literature on strong test ideals by showing that work of Schwede gives an abundance of strong test ideals, and in particular, it recovers the fact that the big test ideal is a strong test ideal as a corollary. These ideas are then explored computationally with the help of an algorithm developed by Katzman and Schwede. A significant portion of the note is dedicated to settling the case of Stanley-Reisner rings. 

\section{Tight closure vs Integral closure}
\subsection{Integral dependence for tight closure elements.}
\begin{Def}
Let $R$ be a Noetherian ring and $n$ a positive integer. We say that $R$ is $n$-tight if, for every ideal $I$ and $x\in  I^*$, $x$ satisfies a degree $n$ integral dependence equation over $I$.  
\end{Def}

It is clear that a ring is $1$-tight if and only if it is weakly $F$-regular. Vraciu's Theorem~\ref{V} coupled with Theorem 2.1 in~\cite{Hu} shows that a local $F$-finite ring $R$ is $n$-tight, where $n$ is the minimal number of generators of the test ideal $\tau$ of $R$. However, the test ideal is in general difficult to compute. In our subsequent sections we provide an algorithmic approach to finding possibly smaller values of $n$ such that $R$ is $n$-tight than the minimal number of generators of the test ideal. The reader should note that we will show in our last section that if $R$ is a Stanley-Reisner ring associated to a simplicial complex $\Delta$, then $R$ is $f_{max}(\Delta)$-tight, where $f_{max}(\Delta)$ is the number of facets of $\Delta$.

Normality plays an important role in tight closure theory, and it is expected that the connection between the tight and integral closures of an ideal is more interesting in the presence of normality. Finding examples of normal rings that are $n$-tight for $n \geq 3$ is not difficult. For example, the cubical cone $R= k[[x,y,z]]/(x^3+y^3+z^3)$ has test ideal equal to $(x, y, z)$ and hence it is $3$-tight. Whether the cubical cone is $2$-tight is in fact an intriguing question, as at present we do not know any normal ring that is 2-tight, but not $1$-tight.

\begin{Q}
\label{cc}
Let $k$ be a field of prime characteristic and $$R= \frac{k[[x, y, z]]}{(x^3+y^3+z^3)}.$$ Is $R$ a $2$-tight ring?
\end{Q}

Regarding this question, it is interesting to note that in the cubical cone any element in the tight closure of an ideal $J$ generated by a system of parameters satisfies a degree two integral dependence relation over $J$. We will explain this statement next, but first we need to state a result due to Corso and Polini.

\begin{Prop}[\cite{CP}]
\label{cp}
Let $\ringR$ be a Cohen-Macaulay local ring, non regular, and $J$ be an ideal generated by a system of parameters in $R$. Let $I = J : \fm$. Then $I^2 = J\cdot I$.
\end{Prop}

\begin{Cor}
Let $\ringR$ be a local Gorenstein ring with test ideal equal to $\fm$. Then for any $x \in J^*$ we have that
$x^2 \in Jx +J^2.$
\end{Cor}

\begin{proof}
Vraciu's result shows that, in a Gorenstein ring, the test ideal is a strong test ideal. So, $\fm \cdot x \subseteq J$ and so $J+ xR \subseteq J : \fm =I$. But $\length_R(I/J) =1$ since $R$ is Gorenstein, so $J +xR = I$, if $x \not\in J$.

For our statement, we can assume $x \not\in J$. Proposition~\ref{cp} implies at once that $(J+xR)^2 =I^2 = J \cdot I = J\cdot (J+xR),$ which gives the desired statement.

\end{proof}

\begin{Rem}
Since the cubical cone $R= k[[x,y,z]]/(x^3+y^3+z^3)$  is Gorenstein and has the test ideal equal to $(x, y, z)$, the above Corollary applies.
\end{Rem}

The absence of normality complicates the relationship between integral and tight closure. One should first investigate the notion we define below, before embarking on a study of $n$-tightness over non normal rings.

\begin{Def}
Let $R$ be a domain and $K$ its fraction field. We say that $R$ is $n$-{\it normal} if every element $x \in K$ integral over $R$ satisfies a degree $n$ integral dependence equation over $R$.
\end{Def}

There are domains that are not $n$-normal, for any $n \geq 1$. In dimension $1$, there is a simple connection between $n$-normality and $n$-tightness.

\begin{Prop}
Let $R$ be a local domain of dimension $1$ with infinite residue field. Fix $n$ a positive integer.  Then $R$ is $n$-tight if and only if $R$ is $n$-normal.
\end{Prop}

\begin{proof}

Assume that $R$ is $n$-normal.

Let $x \in I^*$, where $I$ is a nonzero ideal of $R$. Then $I$ has height one and hence there exists a minimal reduction of $I$ generated by one element $u$.  But then $\overline{(u)} = \overline{I}$ and so $x \in I^* \subseteq \overline{I} = \overline{(u)}$. Therefore, there exists $m \geq 1$ and $a_i\in (u)^{m-i}$ such that $$x^m + a_{m-1}x^{m-1} + \cdots + a_0=0.$$

Write $a_i = b_i u^{m-i}$, with $b_i \in R$ for all $i =0, \ldots, m-1$. Divide by $u^m$ and get $$(x/u)^m + b_{m-1}(x/u)^{m-1} + \cdots + b_0=0.$$

So, $x/u$ has to satisfy a degree $n$ integral dependence equation over $R$. Retracing the steps above, we get that $x$ must satisfy an integral dependence equation of degree $n$ over $(u)$, hence over $I$. In conclusion, $R$ is $n$-tight.

Conversely, assume that $R$ is $n$-tight and let $x=a/b$ integral over $R$. Then $a \in \overline{(b)}= (b)^*$ and hence $a$ satisfies an integral dependence equation over $(b)$ of degree $n$. Dividing by $b^n$ gives an integral dependence equation of $a/b$ of degree $n$ over $R$.

\end{proof}

Vassilev has shown that the conductor ideal $C$ is a strong test ideal in one-dimensional complete domains (see Example 1.11 in~\cite{Hu}). Hence for rings of the form $R=k [[S]]$, where $S$ is a numerical semigroup, we have that $R$ is $n$-normal, for $n$ equal to the minimal number of generators of $C$, and hence $R$ is also $n$-tight by the above observation. This produces examples of rings that are $n$-tight, for any $n$; however, these rings are not normal. This justifies our interest in Question~\ref{cc}, since the cubical cone is normal.

\subsection{Schwede's work and the abundance of strong test ideals}

In this section we approach the topic of the previous section via the theory of test ideals. In the past few years, many interesting connections have been found between concepts originating from tight closure theory and birational geometry. We discuss the notion of Cartier algebra on the ring $R$ and a variant of test ideals that were defined by Schwede in~\cite{Sc-TAMS}. The latter notion will help us refine our understanding of the test ideal of a ring discussed earlier.

\begin{Def} Let $\sC_e = \Hom{R}{R^{1/q}}{R}$. The {\it Cartier algebra} on $R$ is $$\sC = \oplus_{e \geq 0} \fC_e =\oplus_{e \geq 0} \Hom{R}{R^{1/q}}{R}.$$ Note that this is a noncommutative ring and $\sC_0= R$ is not central in $R$, so this object is not an $R$-algebra in the classical sense.

Let $\sD$ be a graded subring of $\sC$ such that $\sD_0=\sC_0 \simeq R$ and
$\sD_e \neq 0$ for some $e >0$. A {\it Cartier algebra pair} on $R$ is a pair of the form $(R, \sD)$.
\end{Def}

Let us remind the reader a few facts of relevance for our paper. Schwede has shown how to associate a test ideal to a Cartier subalgebra on $R$ in~\cite{Sc-TAMS, ScT}.  More precisely, fix $q=p^e$ and let $\phi:R^{1/q} \to R$ be an $R$-linear map. An ideal $J \leq R$ is $\phi$-compatible if $\phi (J^{1/q}) \subseteq J$. The test ideal $\tau(R, \phi)$ was defined by Schwede as the unique smallest $\phi$-compatible ideal that intersects nontrivially with $R^o$. Its existence was proved by Schwede based upon a technical result of Hochster and Huneke on test elements. Similarly, an ideal $J$ is called $\sD$-compatible if $\phi (J^{1/q}) \subseteq J$, for all $\phi \in \sD_e$ and all $e >0$. The test ideal $\tau(R, \sD)$ is the unique smallest $\sD$-compatible ideal that intersects $R^o$ nontrivially.  We will state Schwede's result below, and we direct the reader to ~\cite{Sc-TAMS} and, especially, \cite{ScT} for an overview of these ideas.
\begin{Lma}[Hochster-Huneke, Theorem 5.10 in~\cite{HH-TAMS}; also Lemma 3.6 in~\cite{ScT}]
\label{hh}
Let $R$ and $\phi$ be as above. Then there exists an element $c$ in $R^o$ such that for all $0 \neq d$ there exists $n \in \mathbb{Z}_{>0}$ with $$c \in \phi^n((dR)^{1/p^{ne}}).$$
\end{Lma}

This allows us to state the existence result for test ideals of $R$ and $\phi: R^{1/q} \to R$ and, respectively,  of a subalgebra $\sD$ of $\sC$, mentioned above.

\begin{Thm}[cf. Theorem 3.18 in~\cite{Sc-TAMS}, Lemma 3.8 and Theorem 7.13 in~\cite{ScT}]
\label{form}
Let $R, \phi$ and $c$ be as in the above Lemma. Then $\tau(R, \phi)$ exists and equals

$$\sum_{n \geq 0} \phi^n((cR)^{1/p^{ne}}).$$ The test ideal of an algebra pair $(R, \sD)$ is

$$\tau(R, \sD) = \sum_{e >0} \sum_{\phi \in \sD_e} \tau(R, \phi).$$
\end{Thm}

The result above has a direct consequence for strong test ideals.

\begin{Thm}
\label{main}
Let $\phi: R^{1/q} \to R$ be an $R$-linear map. Then $\tau(R, \phi)$ is a strong test ideal in $R$. Moreover, if $(R, \sD)$ is an algebra pair, then the test ideal $\tau(R, \sD)$ is a strong test ideal.

\end{Thm}

\begin{proof}

According to Theorem~\ref{form}, $$\tau(R, \sD) = \sum_{e >0} \sum_{\phi \in \sD_e} \tau(R, \phi).$$

By using the definition of a strong test ideal, we can see that the sum of two strong test ideals is a strong test ideal, therefore it is enough to show that
$\tau(R, \phi)$ is a strong test ideal in $R$, for any $R$-linear map $\phi: R^{1/q} \to R$.

Let $c$ be a test element for $R$, which exists by Lemma~\ref{hh}. By Theorem~\ref{form}, we can write

$$\tau:=\tau(R, \phi) = \sum_{n \geq 0} \phi^n((cR)^{1/q^n}).$$

Let $J = \cap_{I \leq R}  \tau I : I^*$. Our plan is to show that $J$ is $\phi$-compatible, that is $\phi(J^{1/q^n}) \subseteq J$.

Let $z \in J$. We want $\phi(z^{1/q}) \in J$, or, in other words, $I^* \phi(z^{1/q}) \subset \tau I$, for every ideal $I$ in $R$.

Let $I$ be an ideal in $R$ and $x \in I^*$. Therefore there exists $d \in R^o$ such that  $d x^{qq'} \in I^{[qq']}, $ for all $q=p^{e}, q'=p^{e'}$ with $e,e'\gg 0$. So, 
$d(x^{q})^{q'} \in (I^{[q]})^{[q']}$, which shows that $x^q \in (I\brq)^*$.

But $z \in J$ and then $z (I\brq)^* \in \tau I\brq$, so $zx^q \in \tau I\brq$. Taking the $q$th roots, we have $z^{1/q} x \in \tau^{1/q}I$ and hence

$x \phi(z^{1/q}) = \phi(z^{1/q} x) \subset \phi(\tau^{1/q} I)= I \phi(\tau^{1/q}) \subseteq I \tau.$ This shows that $I^* \phi(z^{1/q}) \subset \tau I$.

Finally, we need to check that $J$ contains an element from $R^o$.

We claim that $c^2 \in J$. To see this note, that $c \in \tau$. Now, for any ideal $I$ in $R$ and $z \in I^*$, we have $c^2 z = c(cz) \subset cI \subset \tau I$.

\end{proof}

This result recovers immediately earlier results of Vraciu and, respectively, Takagi on strong test ideals.

\begin{Cor}[Vraciu]
\label{V}
The test ideal $\tau_b(R)$ is a strong test ideal in $R$.

\end{Cor}

\begin{proof}

The result is immediate by applying our Theorem to the pair $(R, \sC)$ where $\sC$ is the complete algebra of maps on $R$. A consequence of a result by Hara and Takagi, Lemma 2.1 in~\cite{HT}, shows that
$\tau(R, \fC) = \tau_{b}(R)$.
\end{proof}

To state the next result we need the notion of an $\fa^t$-tight closure, where $\fa$ is an ideal of $R$ and $t$ is a positive real number.

\begin{Def}
For an $R$-module $M$ and a submodule $N$ in $M$

$$N^{*, \fa ^t}= \{ x \in M : {\rm there \ exists} \ c\in R^o {\rm \ such \ that} \ cx^q \fa^{\lceil tq \rceil} \in N\brq _M, {\rm for \ all} \ q=p^e \gg0\}.$$

The test ideal $\tau(R, \fa^t) ={\rm Ann}_{R}(0_E^{*, \fa ^t})$ where $E=E_R(R/\fm)$ is the injective hull of the $R$-module $R/\fm$; for details, see~\cite{HT}.

\end{Def}

The next Corollary is due to Takagi, see Proposition 2.2 and Remark 2.3 in~\cite{T} (where one can specialize $\fa= R$).

\begin{Cor}[Takagi]
 Let $\fa$ be an ideal of $R$ and $t \geq 0$. Then $\tau(R, \fa ^t)$ is a strong test ideal in $R$.
\end{Cor}
\begin{proof}
Let $\sC ^{\fa ^t}  = \oplus_{e \geq 0} ( \fa ^{ \lceil t(p^e-1) \rceil })^{1/p^e} \cdot \Hom{R}{R^{1/q}}{R}$ which defines an algebra pair $(R, \sC ^{\fa ^t})$.

Then, as in Exercise 7.9 in~\cite{ScT}, $\tau(R, \sC ^{\fa ^t}) = \tau(R, \fa ^t)$ and now Theorem~\ref{main} gives the result.

\end{proof}

\begin{Cor}
\begin{enumerate}
\item
Let $\phi: R^{1/q} \to R$ be an $R$-linear map. If $\tau(R, \phi)$ is principal, then $R$ is weakly F-regular.

\item
If $(R, \sD)$ is an algebra pair such that  $\tau(R, \sD)$ is principal, then $R$ is weakly F-regular.
\end{enumerate}

\end{Cor}

\begin{proof}
It is obvious that, if $R$ admits a strong test ideal that is principal, then $I^*=I$ for all ideals $I$ in $R$.

\end{proof} 

\subsection{Computations} We illustrate now, with a few examples, how the preceding considerations apply. We have remarked earlier that the number of minimal generators of a strong test ideal represents a uniform bound for the minimal degree of the equation
of integral dependence of an arbitrary element $x \in I^*$ over $I$, where $I$ is an ideal of $R$. Therefore, having a larger class of strong test ideals can give a better bound. In ~\cite{KS}, Katzman and Schwede have produced an algorithm, which was implemented in Macaulay2, that computes all $\phi$-compatible ideals of a surjective $R$-linear map $\phi: R^{1/q} \to R$.

Let $\phi: R^{1/q} \to R$ be a surjective $R$-linear map. In order to compute the test ideal $\tau(R,\phi)$, which is the smallest $\phi$-compatible ideal with respect to inclusion, we have to intersect all the $\phi$-compatible prime ideals, because $\phi$-compatible ideals are closed under radicals and primary decomposition by Proposition 3.2 in \cite{KS}.

By Fedder's Lemma ~\ref{f}, we know that there exists an $S$-linear map $\Phi: S^{1/q} \to S$ which is compatible with $I$ such that $\phi=\Phi/I$ (see also Remark~\ref{trace}). Now, if we want to determine the $\phi$-compatible prime ideals, Lemma 2.4 in ~\cite{KS} tells us that it is enough to determine the $\Phi$-compatible prime ideals that contain $I$ since there is a bijective correspondence between the $\phi$-compatible ideals and the $\Phi$-compatible ideals containing $I$. 

Next, we have to eliminate from this list the set of minimal primes of the ideal $I$, otherwise by intersecting them and modding out the result by the ideal $I$ we obtain the zero ideal. Then the class of the ideal obtained after intersecting the remaining ideals modulo $I$ is the test ideal $\tau(R,\phi)$.


Therefore, we have a concrete way of computing strong test ideals for F-pure rings. The following is an example due to Katzman, and further studied by Katzman and Schwede in ~\cite{KS}, which illustrates this idea. In the following examples, we will generally use the same letter to denote an element of $S$ and its image in $S/I$, when it is harmless to do so, to avoid complicating the notation.

\begin{Exa} Let $k=\mathbb{F}_2$ and let $S= k[[x_1,\ldots, x_5]]$. Let
$\mathcal{I}$ be the ideal generated by the $2 \times 2$ minors of 
\[\begin{pmatrix}x_1&x_2&x_2&x_5\\x_4&x_4&x_3&x_1\end{pmatrix}\]


Consider $R = S/\mathcal{I}$. The ring $R$ is Cohen-Macaulay reduced and two-dimensional. 

Let $\phi: R^{1/2} \to R$ be an $R$-linear map constructed as follows:

Let $S^{1/2} = k\left[\left[x_1^{1/2}, \ldots, x_5^{1/2}\right]\right]$ which is a free $S$-module with basis $\left\{x_1^{\lambda_1/2} x_2^{\lambda_2/2}\cdots x_5^{\lambda_5/2}\right\}_{0\leq\lambda_i\leq 1 }$. Construct $\Phi_S: S^{1/2} \to S$, an $S$-linear map, by sending
 $x_1^{1/2} x_2^{1/2}\cdots x_5^{1/2}$ to $1$ and the other basis elements to zero.

Now fix $z \in (I^{[2]}:_S I)\setminus \fm^{[2]}$. For an element $\overline{s} \in S/I$ we let $\phi( \overline{s} ^{1/2}) =  \Phi_S (z^{1/2} s^{1/2}) \  {\rm modulo} \  I$.

 This defines an $R$-linear map $\phi:R^{1/2} \to R$.

For the choice $z=z_0 = x_1^3x_2x_3+x_1^3x_2x_4 + x_1^2x_3x_4x_5+x_1x_2x_3x_4x_5 + x_1x_2x_4^2x_5+x_2^2x_4^2x_5+x_3x_4^2x_5^2+x_4^3x_5^2$, Katzman and Schwede have applied their algorithm ~\cite{KS} and obtained the list of all $\phi$-compatible prime ideals of $R$.
 The list of $\phi$-compatible prime ideals is as follows
$$R, (x_1, x_4), (x_1, x_4, x_5),$$
$$(x_1 + x_2, x_1^2 + x_4 x_5), (x_1 + x_2, x_2^2 + x_4 x_5), (x_3 + x_4, x_1 + x_2, x_2^2 + x_4 x_5),$$
$$(x_1, x_2, x_5, x_3 + x_4), (x_1, x_2, x_4), (x_1, x_2, x_5), (x_1, x_3, x_4),$$
$$(x_1, x_2, x_3, x_4), (x_1, x_2, x_4, x_5), (x_1, x_3, x_4, x_5), (x_1, x_2, x_3,x_4, x_5).$$
 From this list we can easily identify the unique smallest $\phi$-compatible ideal. Lemma 2.4 in ~\cite{KS} tells us that we have to keep the $\Phi$-compatible prime ideals that contain the ideal $I$ and eliminate the minimal primes of $I$ from this list.

We have that the set of minimal primes of $I$ is given by 
$$Min(I)=\{(x_1 + x_2, x_1^2 + x_4 x_5), (x_1, x_2, x_5), (x_1, x_4, x_3)\}.$$
The list of $\Phi$-compatible prime ideals that contain $I$ and are not in the list of the minimal primes of $I$ is given by
$$(x_1, x_2, x_4, x_5), (x_1, x_2, x_3, x_4, x_5), (x_1, x_2, x_5, x_3 + x_4), (x_1, x_2, x_3, x_4), (x_1, x_3, x_4, x_5).$$
Next, by intersecting them and taking the class modulo $I$ we obtain the test ideal of the pair $(R,\phi)$
$$\tau(R, \phi) = (x_1, x_2 x_5, x_3 x_4 + x_4^2).$$

Therefore, this ring is $3$-tight, so every element $x$ belonging to $I^*$ satisfies a degree $3$ equation of integral dependence over $I$. 

\medskip
We found two more elements $z_1, z_2$ contained in $(I^{[2]}:_S I)\setminus \fm^{[2]}$, such that $z_0, z_1, z_2$ generate $(I^{[2]}:_S I)$.

The first one is 
$z_1 = x_1x_2x_3^2x_5 + x_1x_2x_4^2x_5 + x_1^3x_2x_3 + x_1^3x_2x_4 + x_1^2x_3x_4x_5+x_1x_2x_3x_4x_5 + x_1x_2x_4^2x_5+x_2^2x_4^2x_5+x_3x_4^2x_5^2+x_4^3x_5^2$. 

Since $z_1\in(I^{[2]}:_S I)\setminus \fm^{[2]}$, this element defines an $R$-linear map $\phi_1:R^{1/2} \to R$, given by $\phi_1( \overline{s} ^{1/2}) =  \Phi_S (z_1^{1/2} s^{1/2}) \  {\rm modulo} \  I$, for all $\overline{s} \in S/I$.
We ran the algorithm ~\cite{KS} and  we obtained the list of all $\phi_1$-compatible prime ideals of $R$ as follows

$$R, (x_1 + x_2, x_3 + x_4, x_1^2 + x_4 x_5), (x_3 + x_4, x_2, x_1, x_5),$$
$$  (x_1, x_2, x_5), (x_1, x_3, x_4), (x_1, x_2, x_3, x_4),$$
$$ (x_1, x_3, x_4, x_5), (x_1, x_2, x_3, x_5), (x_1, x_2, x_3, x_4, x_5).$$
The list of $\Phi$-compatible prime ideals that contain $I$ and are not in the list of the minimal primes of $I$ is given by
$$(x_3 + x_4, x_2, x_1, x_5), (x_1, x_2, x_3, x_4),  (x_1, x_2, x_3, x_4, x_5), $$
$$(x_1, x_2, x_3, x_5), (x_1, x_3, x_4, x_5).$$
Next, by intersecting them and taking the class modulo $I$ we obtain the test ideal of the pair $(R,\phi_1)$
$$\tau(R, \phi_1) = (x_1, x_4 x_5, x_3 x_5 , x_2 x_5, x_2 x_4, x_3^2 + x_3 x_4, x_2 x_3).$$

The second one is 
$z_2 = x_1^3 x_3 x_4 + x_1 x_2^2 x_3 x_4 + x_1^3x_2x_3 + x_1^3x_2x_4 + x_1^2x_3x_4x_5+x_1x_2x_3x_4x_5 + x_1x_2x_4^2x_5+x_2^2x_4^2x_5+x_3x_4^2x_5^2+x_4^3x_5^2$.

Since $z_2\in(I^{[2]}:_S I)\setminus \fm^{[2]}$, this element defines an $R$-linear map $\phi_2:R^{1/2} \to R$, given by $\phi_2( \overline{s} ^{1/2}) =  \Phi_S (z_1^{1/2} s^{1/2}) \  {\rm modulo} \  I$, for all $\overline{s} \in S/I$ .
We ran the algorithm ~\cite{KS} and  we obtained the list of all $\phi_2$-compatible prime ideals of $R$ as follows

$$R, (x_1 + x_2, x_3 + x_4, x_2^2 + x_4 x_5), (x_3 + x_4, x_2, x_1, x_5),$$
$$  (x_1, x_2, x_5), (x_1, x_3, x_4), (x_1, x_2, x_3, x_4),$$
$$ (x_1, x_2, x_4, x_5), (x_1 + x_2, x_1^2 + x_4 x_5), (x_1, x_2, x_3, x_4, x_5),$$
$$(x_1, x_2, x_4), (x_1, x_4), (x_1, x_4, x_2 + x_5), (x_1, x_3, x_4, x_2 + x_5).$$
The list of $\Phi$-compatible prime ideals that contain $I$ and are not in the list of the minimal primes of $I$ is given by
$$(x_3 + x_4, x_2, x_1, x_5), (x_1, x_2, x_3, x_4),  (x_1, x_2, x_3, x_4, x_5), $$
$$(x_1, x_2, x_4, x_5), (x_1, x_3, x_4, x_2 + x_5).$$
Next, by intersecting them and taking the class modulo $I$ we obtain the test ideal of the pair $(R,\phi_2)$
$$\tau(R, \phi_2) = (x_1, x_4 x_5, x_3 x_5 , x_2 x_3, x_2 x_4, x_4^2 + x_3 x_4, x_2^2 + x_2 x_5).$$

Hence, both $\tau(R, \phi_1)$ and $\tau(R, \phi_2)$ have seven minimal generators. 
\end{Exa}
\begin{Exa}
Let $k=\mathbb{F}_2$ and $S=k[[x_1,x_2,x_3,x_4]]$. Let $\mathcal{I}=(x_1 x_3, x_1 x_4, x_2 x_3, x_2 x_4)$ and $R=S/\mathcal{I}$. 
Let $\phi: R^{1/2} \to R$ be an $R$-linear map constructed as follows:

Let $S^{1/2} = k\left[\left[x_1^{1/2}, \ldots, x_4^{1/2}\right]\right]$ which is a free $S$-module with basis $\left\{x_1^{\lambda_1/2} x_2^{\lambda_2/2} x_3^{\lambda_3/2} x_4^{\lambda_4/2}\right\}_{0\leq\lambda_i\leq 1}$. Construct $\Phi_S: S^{1/2} \to S$, an $S$-linear map, by sending
 $x_1^{1/2} x_2^{1/2} x_3^{1/2} x_4^{1/2}$ to $1$ and the other basis elements to zero.

Let $z=x_1 x_2 x_3 x_4$ an element contained in $(I^{[2]}:I)\setminus \fm^{[2]}$. The choice of the element $z$ guarantees that the map $\phi$ is surjective from Fedder's Lemma.
By applying the algorithm of Katzman and Schwede ~\cite{KS}, we will get the list of $\phi$-compatible primes 
$$R, (x_4), (x_4, x_3), (x_4, x_3, x_2), (x_4, x_3, x_1), (x_4, x_3, x_2, x_1),$$
$$(x_4, x_2), (x_4, x_2, x_1), (x_4, x_1), (x_3), (x_3, x_2), (x_3, x_2, x_1),$$
$$(x_3, x_1), (x_2), (x_2, x_1), (x_1).$$
Using this list, one can obtain the unique smallest $\phi$-compatible ideal. 
The set of minimal primes of $I$ is $Min(I)=\{(x_1, x_2),(x_3, x_4)\}$. 

The $\Phi$-compatible prime ideals that contain the ideal $I$ and are not minimal primes of $I$ are the following
$$ (x_4, x_3, x_2), (x_4, x_3, x_2, x_1), (x_4, x_3, x_1), (x_4, x_2, x_1), (x_3, x_2, x_1).$$
After intersecting them in $R=S/I$ we obtain the test ideal of the pair $(R,\phi)$ is 
$$\tau(R, \phi) = (x_3 x_4, x_1 x_2).$$

Therefore, this ring is $2$-tight, hence every element $x$ belonging to $I^*$ satisfies a degree $2$ equation of integral dependence over $I$. 

\medskip
We notice that the number of generators of $\tau(R, \phi)$ is actually the number of facets of the simplicial complex $\Delta$ associated to the square-free monomial ideal $I$. In the next section, Corollary ~\ref{facets} will show that this happens for all Stanley-Reisner rings.
\end{Exa}

Next we will consider, from our perspective, a well-known example of an F-rational, F-pure ring that is not weakly F-regular from ~\cite{HH-TAMS}.

\begin{Exa}
Let $S=K[[X,Y,Z]]/(F)$, where $K$ is an algebraically closed field of characteristic $p=7$ and $F = X^3 - YZ(Y+Z)$. Let $\omega$ be a primitive cube root of unity in $K$. Let $G = \{1, \omega, \omega^2\}$ act $K$-linearly on $S$ so as to send the images $x, y, z$ of $X, Y, Z$ to $x, \omega y, \omega z$. Let $R = S^G$ be the fixed ring of this action, which is generated over $K$ by $x, y^3, y^2z$, and $z^3$.

We obtained that $R$ is isomorphic to $K[[a,b,c,d]]/I$, where $I=(c^3-b^2d, a^3c-c^2-bd, a^3b-bc-c^2, a^6-c^2-2bd-cd)$.

Using Macaulay 2, we found three elements in $(I^{[7]}:I)\setminus m^{[7]}$ generating $(I^{[7]}:I)$. Since these elements have around 250 terms each, we will not list them here.

We ran the algorithm ~\cite{KS} for each one of these elements and we obtained that the maximal ideal is the only $\phi$-compatible prime ideal of $R$, where $\phi$ is the map associated to the respective element. Therefore, the test ideal of the pair $(R, \phi)$ is the maximal  ideal in each of the three cases. We conclude that $R$ is therefore $4$-tight.
\end{Exa}

\medskip
The reader should be aware that current algorithmic tools are sometimes insufficient to compute such strong test ideals. The algorithm implemented by Katzman and Schwede needs the map $\phi$ to be surjective (so $R$ needs to be $F$-pure), and, even in the presence of $F$-purity, the computer can fail to produce the generators of ideals of type $I\brq :I$.

\section{Stanley-Reisner rings}
Let $k$ be a perfect field of characteristic $p$, $S=k[[x_1,\ldots,x_n]]$ be the formal power series ring in $n$ variables over $k$ and $q=p^e$, for $e\geq 0$.
Let $I\leq S$ be a square-free monomial ideal in $S$ and $R=S/I$.

Let $V=[n]=\{1,\ldots,n\}$ and $\Delta$ a simplicial complex on the vertex set $V$ of dimension $d-1\geq 0$. We denote the number of $i$-dimensional faces of $\Delta$ by $f_i$. We have $f_0=n$ and $f_{-1}=1$, since the empty face is a face of dimension $-1$ of any non-empty simplicial complex.

The $d$-tuple
$$f(\Delta) = (f_0, f_1,\ldots, f_{d-1})$$
is called the \textbf{\textit{$f$-vector of $\Delta$}}.
Let $f_{max}(\Delta)$ be the number of facets of the simplicial complex $\Delta$.

We can associate to the ideal $I$ a simplicial complex $\Delta$ on $[n]$ such that $\Delta$ contains a face $F$ if and only if $x_F\notin I$, where $x_F=\prod_{i\in F} x_i$ and $I=(x_F: F\notin \Delta)$.

\begin{Rem}
Let $\Delta$ be a simplicial complex on $[n]$. Then the primary decomposition of the Stanley-Reisner ideal associated to $\Delta$ is given by
$$I_{\Delta}=\bigcap_{F\in\mathcal{F}(\Delta)}P_{F^{c}},$$
where $\mathcal{F}(\Delta)$ is the set of the facets of $\Delta$, $F^{c}=[n]\setminus F$ and $P_F=(x_i: i\in F)$. 
\end{Rem}
Let $\alpha=(\alpha_1,\ldots,\alpha_n)\in \{0,1\}^{n}$. We define the support of $\alpha$ as $supp(\alpha)=\{i : \alpha_i\neq 0\}$.

Let $S=k[[x_1,\ldots,x_n]]$, where $k$ is a perfect field of characteristic $p$ and $q=p^e$ for $e\geq 0$. Then, $S^{1/q}$ is a free $S$-module with basis $\left\{x_1^{\lambda_1/q}\cdots x_n^{\lambda_n/q}\right\}_{0\leq \lambda_i\leq q-1}$. The map $\Phi_S: S^{1/q}\to S$ that sends the element $x_1^{{(q-1)}/q}\cdots x_n^{{(q-1)}/q}$ to $1$ and all the other basis elements to zero is called \textbf{\textit{the trace map}}.

\begin{Rem}
\label{trace}
Let $S=k[[x_1,\ldots,x_n]]$, where $k$ is a perfect field of characteristic $p$ and $q=p^e$ for $e\geq 0$. Then $\Hom{S}{S^{1/q}}{S}$ is a free $S^{1/q}$-module with generator $\Phi_S$. Therefore, for every $S$-linear map $\Phi: S^{1/q}\to S$, there is $z \in S$ such that $\Phi(s)=\Phi_S(z^{1/q}s)$, for every $s\in S^{1/q}$.
\end{Rem}
\begin{Thm}[Fedder's Lemma, cf. Lemma 1.6 in~\cite{Fed}]
\label{f}

Let $S=k[[x_1,\ldots,x_n]]$, where $k$ is a perfect field and $R=S/I$ for some ideal $I\leq S$. If $\phi: R^{1/q} \to R$ is any $R$-linear map, then there exists an $S$-linear map $\Phi: S^{1/q} \to S$ which is compatible with $I$ such that $\phi=\Phi/I$.

Moreover,
$\phi$ is surjective if and only if $z\notin \fm^{[q]}$, where $\Phi(s)=\Phi_S(z^{1/q}s)$ and $\Phi_S$ is the trace map on $S$. 
Furthermore, there exists an isomorphism 
$$\Hom{R}{R^{1/q}}{R}\cong \displaystyle\frac{I^{[q]}:I}{I^{[q]}}.$$
\end{Thm}
\begin{Prop}[Proposition 3.2 in~\cite{Boix}]
\label{b}
Let $r$ be a positive integer. Consider  $\alpha_k=(\alpha_{k1},\ldots,\alpha_{kn})\in \{0,1\}^n$ and $I_{\alpha_k}=(x_i: i\in supp(\alpha_k))$, for every $1\leq k\leq r$. Assume that $\alpha_k$, $ 1\leq k \leq r$, are distinct vectors.  Set $I = I_{\alpha_1}\cap I_{\alpha_2}\cap\cdots\cap I_{\alpha_r}$. Then
$$(I^{[q]}:_S I)= (I_{\alpha_1}^{[q]}+(x^{\alpha_1})^{q-1})\cap\cdots\cap(I_{\alpha_r}^{[q]
}+(x^{\alpha_r})^{q-1}),$$
where $x^{\alpha_k}=x_1^{\alpha_{k1}}\cdots x_n^{\alpha_{kn}}$.
\end{Prop}
\begin{Cor}\label{surj}

Let $I\subseteq S$ be a square-free monomial ideal and $R=S/I$.
Then, $z=(\prod_{i=1}^n x_i)^{q-1}\in (I^{[q]}:_S I)\setminus \fm^{[q]}$. Therefore, $z=(\prod_{i=1}^n x_i)^{q-1}$ defines an $R$-linear surjective map  $\phi: R^{1/q}\to R$, $\phi=\Phi/I$ with $\Phi(s)=\Phi_S((\prod_{i=1}^n x_i)^{q-1/q}s)$, for all $s\in S^{1/q}$.
\end{Cor}
\begin{proof}


Since $I$ is a square-free monomial ideal, the minimal primary decomposition of $I$ can be written as $I = I_{\alpha_1}\cap I_{\alpha_2}\cap\cdots\cap I_{\alpha_r}$, where $\alpha_k=(\alpha_{k1},\ldots,\alpha_{kn})\in \{0,1\}^n$, $1\leq k\leq r$, are distinct vectors, and $I_{\alpha_k}=(x_i: i\in supp(\alpha_k))$, for every $1\leq k\leq r$. 

 By using Proposition~\ref{b}, we obtain that $lcm((x^{\alpha_1})^{q-1},(x^{\alpha_2})^{q-1},\ldots,(x^{\alpha_r})^{q-1})$ is an element contained in $(I^{[q]}:_S I)$ that is not in $\fm^{[q]}$.

But $lcm((x^{\alpha_1})^{q-1},(x^{\alpha_2})^{q-1},\ldots,(x^{\alpha_r})^{q-1})$ divides $(\prod_{i=1}^n x_i)^{q-1}$ because $x^{\alpha_1},\ldots, x^{\alpha_r}$ are square-free monomials. Hence, $(\prod_{i=1}^n x_i)^{q-1}\in (I^{[q]}:I)\setminus \fm^{q}$.

Therefore, by Theorem~\ref{f} the $R$-linear map $\phi: R^{1/q}\to R$, given by $z=(\prod_{i=1}^n x_i)^{q-1}$ is a surjective map. 
\end{proof}
\begin{Prop}[Corollary 1.5 in~\cite{BK}]\label{bk}
Let $\Phi: S^{1/q}\to S$ be an $S$-linear map and $z\in S$ such that $\Phi(s)=\Phi_S(z^{1/q}s)$, for every $s\in S^{1/q}.$ Let $J\subseteq S$ be an ideal in $S$. Then $J$ is $\Phi$-compatible if and only if $J\subseteq (J^{[q]}:_S z)$.
\end{Prop}
\begin{Def}
Let $K\subseteq S$ be an ideal in $S$ and $q=p^e$, for $e\geq 0$. Then $I_e(K)$ denotes the smallest ideal $I$ such that $I^{[q]}\supseteq K$. The ideal $I_e(K)$ is called the $e${\it-th root ideal} of $K$. 
\end{Def}
We have that the following elementary properties of the $e$-th root ideals hold.
\begin{Prop}[Proposition 1.3 in~\cite{BK}]
\label{root}
Let $K_1,\ldots, K_s\subseteq S$ ideals in $S$. Then the following statements hold:
\item[(a)]
$I_e(\displaystyle\sum_{i=1}^s K_i)=\displaystyle\sum_{i=1}^s I_e(K_i);$
\item[(b)]
Let $h\in S$ and write
$$h=\sum_{0\leq a_1,\ldots,a_n \leq q-1, \  a=(a_1, \ldots, a_n)} h_{a}^{q}x_1^{a_1}\cdots x_n^{a_n}.$$
Then $I_e(h)$ is the ideal generated in $S$ by the elements $h_{a}$ appearing in the expression above.
\end{Prop}
\begin{Prop}\label{comp}
Let $S=k[[x_1,\ldots,x_n]]$, where $k$ is a perfect field of characteristic $p$. Let $\Phi: S^{1/q}\to S$ be given by $\Phi(s)=\Phi_S(z^{1/q}s)$, for every $s\in S^{1/q}$ and $z=(\prod_{i=1}^n x_i)^{q-1}$.  The set of $\Phi$-compatible prime ideals consists of the set of ideals generated by variables, that is $(x_{i_1},\ldots, x_{i_k})$, where $1\leq i_1,\ldots,i_k\leq n$.
\end{Prop}
\begin{proof}
In order to see that the ideals generated by variables are $\Phi$-compatible we will use Proposition ~\ref{bk}. For example, if we consider the ideal $(x_{i_1},\ldots,x_{i_k})$, it is easy to see that $(x_1\cdots x_n)^{q-1}(x_{i_1},\ldots,x_{i_k})\subseteq (x_{i_1}^q,\ldots,x_{i_k}^q)$. By using Proposition ~\ref{bk}, we obtained that $(x_{i_1},\ldots,x_{i_k})$ is $\Phi$-compatible.

On the other hand, we have to show that the ideals generated by variables are the only $\Phi$-compatible prime ideals. In order to prove this, it is enough to show that if an ideal, say $J$, is a prime $\Phi$-compatible ideal, then $J$ is monomial, since every prime monomial ideal is an ideal generated by variables.

Let $J$ be a $\Phi$-compatible prime ideal and let $f$ be a polynomial in $J$. We let $f=\sum_{i=1}^{r}f_i$ be the decomposition of $f$ as a sum of monomials. We have to show that each monomial component $f_i$ of $f$ is contained in $J$.

Since $zf\in J^{[q]}$, then $I_e(zf)\subseteq J$, where $z=(x_1\cdots x_n)^{q-1}$. But by Proposition ~\ref{root} (a), $I_e(zf)=\sum_{i=1}^r I_e(zf_i)$. Moreover, Proposition~\ref{root} (b) gives that $f_i\in I_e(zf_i)$, for $1\leq i\leq r$. Hence, each $f_i$ is contained in $J$. Therefore, $J$ is a monomial prime ideal.

To sum up, all the $\Phi$-compatible ideals are the ideals generated by variables.
\end{proof}
\begin{Prop}\label{tau}
Let $I\subseteq S$ be a square-free monomial ideal and $R=S/I$. Let $\phi: R^{1/q}\to R$ be the $R$-linear map given by $z=(\prod_{i=1}^n x_i)^{q-1}$, that is, $\phi=\Phi/I$ with $\Phi(s)=\Phi_S((\prod_{i=1}^n x_i)^{(q-1)/q}s)$, for all $s\in S^{1/q}$. Then the test ideal associated to the pair $(R,\phi)$ is given by
$$\tau(R,\phi)=(x_F: F\in\mathcal{F}(\Delta)),$$
where $\Delta$ is the simplicial complex associated to the ideal $I$.
\end{Prop}
\begin{proof}
Given $\phi:R^{1/q}\to R$ an $R$-linear map, there exists an $S$-linear map $\Phi: S^{1/q} \to S$ which is compatible with $I$ such that $\phi=\Phi/I$ by Theorem ~\ref{f}, where $\Phi(s)=\Phi_S(z^{1/q}s)$ and $\Phi_S$ is the trace map on $S$. Moreover, $\phi$ is surjective if and only if $z\notin \fm^{[q]}$, 

 But according to Corollary ~\ref{surj}, $z=(\prod_{i=1}^n x_i)^{q-1}$ defines an $R$-linear surjective map $\phi: R^{1/q}\to R$, that is, $\phi=\Phi/I$ with $\Phi(s)=\Phi_S((\prod_{i=1}^n x_i)^{(q-1)/q}s)$ for all $s\in S^{1/q}$.
Using Lemma 2.4 in ~\cite{KS}, we have that there is a bijective correspondence between the $\phi$-compatible ideals and the $\Phi$-compatible ideals containing $I$.

Proposition ~\ref{comp} gives the list of $\Phi$-compatible prime ideals. We want to compute $\tau(R,\phi)$, which is the smallest $\phi$-compatible ideal with respect to inclusion. Since, in an $F$-pure ring, the $\phi$-compatible ideals are closed under primary decomposition, we need to intersect all the $\phi$-compatible prime ideals. By Lemma 2.4 in ~\cite{KS}, to determine the list of all $\phi$-compatible prime ideals,  we first find the $\Phi$-compatible prime ideals that contain the ideal $I$. Then we remove the minimal primes of $I$ from the list given by Proposition ~\ref{comp}. After this, $\tau(R,\phi)$ is the image of the ideal obtained after intersecting all these remaining ideals modulo $I$.


Consider now the simplicial complex $\Delta$ associated to the ideal $I$. Let $\mathcal{F}(\Delta)=\{F_1,\ldots,F_m\}$ denote the set of facets of $\Delta$ and 
$$I=I_{\Delta}=\bigcap_{F\in\mathcal{F}(\Delta)}P_{F^{c}}$$
the primary decomposition of the ideal $I$.

So we have that the set of minimal primes of $I$ is $Min(I)=\{P_{F^{c}}\}$. Proposition ~\ref{comp} tells us that the set of $\Phi$-compatible prime ideals consists of all the ideals generated by variables. 
Hence, the set of $\Phi$-compatible prime ideals that contain $I$ and are not in the set of minimal primes of $I$ are the ideals 
$$(P_{F_j^c}, x_i: i\in F_j),$$
for every $1\leq j\leq m$.
Therefore, by intersecting them, we obtain 
$$\bigcap_{j=1}^m (P_{F_j^c}, \prod_{i\in F_j} x_i)=\bigcap_{j=1}^m (P_{F_j^c}, x_{F_j}).$$

Now, we obtain the test ideal $\tau(R,\phi)$ by taking the intersection $$\bigcap_{j=1}^m (P_{F_j^c}, x_{F_j})$$ modulo the ideal $I$.
Since $I=I_{\Delta}=(x_F: F\notin \Delta)$, all the monomials in the intersection $$\bigcap_{j=1}^m (P_{F_j^c}, x_{F_j})$$ are killed by modding out by the ideal $I$, except $x_{F_1},\ldots, x_{F_m}$.
Hence, 
 $$\tau(R,\phi)=(x_{F_1},\ldots, x_{F_m}).$$



\end{proof}
\begin{Cor}\label{facets}
 
Let $I\subseteq S$ be a square-free monomial ideal and $R=S/I$. Let $\phi: R^{1/q}\to R$ be the $R$-linear map given by $z=(\prod_{i=1}^n x_i)^{q-1}$, that is, $\phi=\Phi/I$ with $\Phi(s)=\Phi_S((\prod_{i=1}^n x_i)^{(q-1)/q}s)$ for all $s\in S^{1/q}$.

Then the test ideal associated to the pair $(R,\phi)$ is $f_{max}(\Delta)$-generated,
where $\Delta$ is the simplicial complex associated to the ideal $I$.

Therefore, in this ring, every element $x$ belonging to $I^*$ satisfies a degree $f_{max}(\Delta)$ equation of integral dependence over $I$.
\end{Cor} 

\
\bigskip
\section*{Acknowledgements} The authors thank Siang Ng and Yongwei Yao for useful comments and Robin Baidya for a careful reading of the manuscript and writing feedback.

\end{document}